\documentclass[12pt]{amsart}
\usepackage[active]{srcltx}
\usepackage{calc,amssymb,amsthm,amsmath,amscd, eucal,ulem}
\usepackage{alltt}
\usepackage[left=1.35in,top=1.25in,right=1.35in,bottom=1.25in]{geometry}
\RequirePackage[dvipsnames,usenames]{color}

\input{kmacros3.sty}
\input{mabliautoref.sty}
\input{xy}
\xyoption{all}
\usepackage{tikz}

\numberwithin{equation}{theorem}

\newcommand{\F}{\mathbb{F}}

\renewcommand{\m}{\mathfrak{m}}
\renewcommand{\n}{\mathfrak{n}}

\DeclareMathOperator{\depth}{depth}

\usepackage{fullpage}

\usepackage{setspace}
\usepackage{hyperref}

\usepackage{enumerate}

\usepackage{graphicx}

\usepackage[all,cmtip]{xy}
\usepackage{verbatim}

\theoremstyle{theorem}

\newtheorem*{mainthma}{Main Theorem A}
\newtheorem*{mainthmb}{Main Theorem B}
\newtheorem*{mainthmc}{Main Theorem C}

\renewcommand{\O}{\mathcal O}

\begin{document}
\title{The dualizing complex of $F$-injective and Du Bois singularities}
\author{Bhargav Bhatt}
\author{Linquan Ma}
\author{Karl Schwede}
\address{Department of Mathematics\\University of Michigan\\2074 East Hall, 530 Church Street\\Ann Arbor\\MI 48109}
\email{bhattb@umich.edu}
\address{Department of Mathematics\\ University of Utah\\ Salt Lake City\\ UT 84112}
\email{lquanma@math.utah.edu}
\address{Department of Mathematics\\ University of Utah\\ Salt Lake City\\ UT 84112}
\email{schwede@math.utah.edu}

\thanks{The first named author was supported by a NSF Grants DMS \#1501461 and DMS \#1522828 and by a Packard Fellowship.} 
\thanks{The second named author was supported by NSF CAREER Grant DMS \#1252860/1501102 and a Simons Travel Grant.  }
\thanks{The third named author was supported in part by the NSF FRG Grant DMS \#1265261/1501115, NSF CAREER Grant DMS \#1252860/1501102 and a Sloan
  Fellowship.}

\maketitle

\begin{abstract}
Let $(R,\m, k)$ be an excellent local ring of equal characteristic. Let $j$ be a positive integer such that $H_\m^i(R)$ has finite length for every $0\leq i <j$. We prove that if $R$ is $F$-injective in characteristic $p>0$ or Du Bois in characteristic $0$, then the truncated dualizing complex $\tau_{>-j}\omega_R^\mydot$ is quasi-isomorphic to a complex of $k$-vector spaces. As a consequence, $F$-injective or Du Bois singularities with isolated non-Cohen-Macaulay locus are Buchsbaum. Moreover, when $R$ has $F$-rational or rational singularities on the punctured spectrum, we obtain stronger results generalizing \cite{MaF-injectivityandBuchsbaumsingularities} and \cite{IshidaIsolatedDuBoissingularities}.
\end{abstract}

\section{Introduction}

In this paper, we study the dualizing complexes of $F$-injective and Du~Bois singularities.  We prove the following result.

\begin{mainthma}[\autoref{theorem--main theorem on F-injective}, \autoref{corollary--main corollary on DB}]
\label{thma}
Suppose that $(R,\m,k)$ is a Noetherian local ring such that $H^i_{\m}(R)$ has finite length\footnote{Or dually, that $\myH^{-i} \omega_R^{\mydot}$ has finite length for $0 \leq i < j$ (provided a dualizing complex exists). Under mild conditions this is equivalent to saying that the non-Cohen-Macaulay locus on $\mathrm{Spec}(R)$ has codimension $j$.} for every $0\leq i <j$.  Suppose one of the following conditions hold:
\begin{enumerate}
\item $R$ is of characteristic $p$ and is $F$-injective.
\item $R$ is essentially of finite type over a field of characteristic zero and is Du Bois.
\end{enumerate}
Then the truncated dualizing complex $\tau_{>-j}\omega_R^\mydot$ is quasi-isomorphic to a complex of $k$-vector spaces.  Equivalently, $\tau^{< j} \myR \Gamma_{\m}(R)$ is quasi-isomorphic to a complex of $k$-vector spaces.  In particular, these truncated complexes split into a direct sum of their cohomologies.
\end{mainthma}

In the case that $j = \dim R$, the condition proven in the theorem is also known as \emph{Buchsbaum} \cite{StuckradandVogelBuchsbaumringsandapplications}, \cite{SchenzelApplicationsOfDualizingComplexes}.  Buchsbaum should be thought of as a minimal generalization of the Cohen-Macaulay condition.  In the special case of $j = \dim R$, our result says that:
\begin{enumerate}
\item $F$-injective singularities with isolated non-Cohen-Macaulay locus are Buchsbaum, reproving a result of \cite{MaF-injectivityandBuchsbaumsingularities}.
\item Du~Bois singularities with isolated non-Cohen-Macaulay locus are Buchsbaum, generalizing the case of normal isolated singularities \cite{IshidaIsolatedDuBoissingularities} and answering a question of Shunsuke Takagi.
\end{enumerate}

Moreover, in characteristic $p>0$, we define the {\it $*$-truncation} or {\it tight closure truncation} $\tau^{<d,*}\myR\Gamma_\m R$, to be the object $D$ in the derived category of $R$-modules such that we have an exact triangle: $$D \to \myR\Gamma_\m R\to (H_\m^d(R)/0^*_{H_\m^d(R)})[-d]\xrightarrow{+1}.$$ This complex $\tau^{<d,*}\myR\Gamma_\m R$ has the property that its lower cohomologies are $H_\m^i(R)$ while its top cohomology is $0^*_{H_\m^d(R)}$.  Indeed, this complex is just Matlis dual to $C$ where
\[
\tau(\omega_R)[d] \to \omega_R^{\mydot} \to C \xrightarrow{+1}
\]
is a triangle and $\tau(\omega_R)$ is the parameter test module \cite{Smithtightclosureofparameterideals}, \cite{BlickleSchwedeTuckerF-singularitiesvialterations}.  Our second main result is the following:

\begin{mainthmb}[\autoref{theorem--main theorem on F-injective tight closure}]
\label{thmb}
Let $(R,\m, k)$ be an excellent equidimensional Noetherian local ring that is $F$-rational on the punctured spectrum. If $R$ is $F$-injective then $\tau^{<d,*}\myR\Gamma_\m R$ is quasi-isomorphic to a complex of $k$-vector spaces. In particular, $\tau^{<d,*}\myR\Gamma_\m R$ splits into a direct sum of its cohomologies.
\end{mainthmb}

We note that if $\tau^{<d,*}\myR\Gamma_\m R$ splits into its cohomologies, then so does $\tau^{<d}(\tau^{<d,*}\myR\Gamma_\m R)=\tau^{<d}\myR\Gamma_\m R$. Hence Theorem B generalizes Theorem A in the case $j=d=\dim R$ (and $R$ is $F$-rational on the punctured spectrum), because it encodes information about the top local cohomology. We also mention that the full complex $\myR\Gamma_\m R$ never splits into its cohomologies (in the derived category) unless $R$ is already Cohen-Macaulay, see \autoref{corollary--full dualizing complex never splits}.

In characteristic zero, we also obtain an analogous result that in some ways is even stronger since it holds even without a Du~Bois / $F$-injective hypothesis.  Indeed if $\pi : W \to X$ is a resolution of singularities of some reduced scheme $X$, then we have the exact triangle
\[
\pi_* \omega_W [d] \to \myR \sHom_{\O_X}(\DuBois{X}, \omega_X^{\mydot}) \to C \xrightarrow{+1}
\]
When $X$ has Du~Bois singularities $\DuBois{X} \qis \O_X$ and the complex $C$ is analogous to the $C$ described above in characteristic $p > 0$.  However, even without this assumption we have the following theorem.
\begin{mainthmc}[\autoref{thm.MainTheoremOnDuBois}, \autoref{corollary--main corollary on DB punctured rational}]
With notation as above, suppose that $X = \Spec R$ is the spectrum of a local ring $(R,\bm, k)$ of essentially finite type over a field of characteristic zero and $\myH^{-i}(C)$ has finite length for all $i$ (this last condition happens automatically if $R$ has rational singularities on the punctured spectrum).  Then $C$ is quasi-isomorphic to a complex of $k$-vector spaces and in particular, it splits into a direct sum of its cohomologies.
\end{mainthmc}
\vskip 9pt
\noindent
\subsection{This paper's history}  The first version of this paper that appeared on the arXiv was written only by the second and third authors.  Shortly after it appeared the first author contacted the second and third authors with an alternate proof of Main Theorem A (at that time, we required the residue field to be perfect).  Together, all the authors generalized the strategy of the new proof to the case where the residue field is not necessarily perfect.  The original proof can now be found in \autoref{subsec.altProofs}.  The first author also suggested a more conceptual proof of \autoref{corollary--full dualizing complex never splits} which lead to the more generalized statement \autoref{proposition--local cohomology complex are indecomposable}.

\vskip 9pt
\noindent
\emph{Acknowledgements:}  We would like to thank Shunsuke Takagi for several useful discussions and for bringing these questions to our attention.  We would also like to thank S\'andor Kov\'acs, Zsolt Patakfalvi and Sean Sather-Wagstaff for useful discussions.  Finally we thank Lance Miller, Kazuma Shimomoto and the referees for comments on previous drafts of this paper.

\section{Preliminaries}
Throughout this paper, all rings are commutative with multiplicative identity 1 and all schemes are separated.  In most cases, we work with Noetherian rings and we also assume all Noetherian rings and schemes have dualizing complexes.  When working with dualizing complexes on Noetherian local rings, we always assume that they are normalized, which means the first nonzero cohomology is in degree $-\dim R$. In characteristic $p > 0$, we will typically assume that our Noetherian rings are \emph{$F$-finite}, which means that the Frobenius morphism is a finite morphism.  We recall that $F$-finite Noetherian rings are always excellent \cite{KunzOnNoetherianRingsOfCharP} and always have dualizing complexes \cite{Gabber.tStruc}.

For a ring $R$, we use $D(R)$ to denote the derived category of $R$-modules.  For $I \subseteq R$ an ideal, we write $D_I(R)$ for the full subcategory of $D(R)$ spanned by objects $K \in D(R)$ such that $K$ restricts to $0$ in $D(U)$, where $U = \mathrm{Spec}(A) - \mathrm{Spec}(R/I)$

We remind the reader of the definitions of $F$-injective and Du~Bois singularities.
\begin{definition}[$F$-injective singularities, \cite{FedderFPureRat}]
Suppose $(R, \bm)$ is an $F$-finite Noetherian local ring of characteristic $p > 0$.  Then $R$ is \emph{$F$-injective} if $H^i_{\bm}(R) \xrightarrow{F} H^i_{\bm}(F_* R)$ is injective for all $i \geq 0$.  Or dually that $\myH^{-i} F_* \omega_R^{\mydot} \to \myH^{-i} \omega_{R}^{\mydot}$ is surjective for all $i$.
\end{definition}

\begin{definition}[Du~Bois singularities, \cite{DuBoisMain,KovacsSchwedeDuBoisSurvey}]
\label{def.DuBoisSings}
Suppose that $X$ is a scheme essentially of finite type over a field of characteristic zero.  We say that $X$ is \emph{Du~Bois} if the canonical map $\O_X \to \DuBois{X}$ is a quasi-isomorphism.  A quick way to define $\DuBois{X}$ is as follows.  If $X \subseteq Y$ can be embedded in a smooth scheme and $\pi : \tld Y \to Y$ is a log resolution of $X \subseteq Y$, or an embedded resolution of $X \subseteq Y$ such that the reduced exceptional divisor of $\pi$ is SNC and intersects the strict transform $\tld X$ in a SNC divisor.  In either case, let $\overline{X} = (\pi^{-1}(X))_{\red}$.  Then $\DuBois{X} = \myR \pi_* \O_{\overline{X}}$, see \cite[Theorem 6.4]{KovacsSchwedeDuBoisSurvey}
\end{definition}

We assume that the readers are already familiar with some of the basic properties of $F$-injective and Du~Bois singularities, see for instance \cite{FedderFPureRat,KovacsSchwedeDuBoisSurvey}.

The following result was previously used to give another proof of Schenzel's homological criterion of Buchsbaum singularities \cite{SchenzelApplicationsOfDualizingComplexes}.

\begin{lemma}[Proposition 2.4.3 in \cite{StuckradandVogelBuchsbaumringsandapplications}]
\label{lemma--complex of vector spaces}
Let $(R,\m, k)$ be a Noetherian local ring and let $f^\mydot$: $K^\mydot\to L^\mydot$ be a homomorphism of complexes of $R$-modules such that
\begin{enumerate}
\item $K^\mydot$ is a complex of $k$-vector spaces.
\item $\myH^i(f^\mydot)$: $\myH^i(K^\mydot)\to \myH^i(L^\mydot)$ is a surjective homomorphism for every $i$.
\end{enumerate}
Then $L^\mydot$ is quasi-isomorphic to a complex of $k$-vector spaces.
\end{lemma}

In this article we will need the following slight generalization of the dual version of this lemma in the derived category and for potentially non-Noetherian rings.

\begin{lemma}
\label{lem:ComplexofVectorSpaces}
Let $(R,\m,k)$ be a local ring. Fix $K \in D(R)$, $L\in D(k)$. Assume there exists a map $f$: $K \to L$ in $D(R)$ such that the induced maps $\myH^i(K) \to \myH^i(L)$ are injective. Then $K$ comes from an object of $D(k)$ under the forgetful functor $D(k) \to D(R)$, i.e., $K$ is quasi-isomorphic to a complex of $k$-vector spaces.
\end{lemma}
\begin{proof}
We can choose an isomorphism $L \simeq \oplus_i \myH^i(L)[-i]$ in $D(k)$. As the map $f$: $\myH^i(K) \to \myH^i(L)$ is injective, we can choose a decomposition $\myH^i(L) \simeq f(\myH^i(K)) \oplus Q_i$. These choices give a map $L \simeq \oplus_i \myH^i(L)[-i] \to \oplus_i f(\myH^i(K))[-i] =: K'$ in $D(k)$. The composite $K \to L \to K'$ in $D(R)$ is a quasi-isomorphism by construction, so $K \simeq K'$ comes from $D(k)$.
\end{proof}

\section{$F$-injective singularities}
\label{sec:F-injective}

In this section, we prove our main results in characteristic $p$.  
The proofs in this section rely crucially on the following observation concerning perfect rings:

\begin{proposition}
\label{prop:PerfectFullyFaithfulRoS}
Let $A \to B$ be a surjection of perfect $\F_p$-algebras. Let $K \in D^b(A)$ be a complex such that each $H^i(K)$ is a $B$-module. Then $K \simeq K \otimes_A^\myL B$ via the canonical map, and thus $K$ comes from $D^b(B)$ via the forgetful functor $D^b(B) \to D^b(A)$.
\end{proposition}
\begin{proof}
We must check that the canonical map $K \to K \otimes_A^\myL B$ is an isomorphism for $K$ as above. This assertion is stable under exact triangles, so we reduce to $K = M[0]$ being a $B$-module $M$ placed in degree $0$. But then $K \otimes_A^\myL B \simeq M[0] \otimes_B^\myL (B \otimes_A^\myL B)$, so the claim follows from \cite[Lemma 5.10]{BhattScholzeWittGr}, which implies that $B \otimes_A^\myL B \simeq B$ via the multiplication map.
\end{proof}

\begin{remark}
Let $A \to B$ be a surjection of perfect $\F_p$-algebras. Using the same method used to prove \autoref{prop:PerfectFullyFaithfulRoS}, one can show the following: the forgetful functor $D(B) \to D(A)$ is fully faithful, and the essential image comprises those $K \in D(A)$ such that each $H^i(K)$ is a $B$-module. We do not prove this here as we do not need it.
\end{remark}

Using \autoref{prop:PerfectFullyFaithfulRoS}, we obtain a criterion for when a ``finite length $\phi$-complex'' is pushed forward from the closed point:

\begin{lemma}
\label{lem:PhiModuleCriterion}
Let $(R,\m,k)$ be a local ring of characteristic $p > 0$ with absolute Frobenius $F:R \to R$. Fix $K \in D(R)$ such that each $H^i(K)$ has finite length. Assume we are given a map $\phi_K:K \to F_* K$ which is injective on each $H^i$. Then $K$ comes from $D(k)$.
\end{lemma}
\begin{proof}
Let $R_\infty$ be the perfection of $R$, so $R_\infty \simeq \varinjlim_e F^e_* R$ is the direct limit of copies of $R$ along (twists of) the Frobenius map $F:R \to F_* R$. In particular, $R_\infty$ is a local ring (as Frobenius is a homeomorphism), is perfect (by construction), has maximal ideal $\m_\infty := \varinjlim_e F^e_* \m$ (with the same transition maps as before), and has residue field $k_\infty$, the perfection of $k$. Define $K_{\infty}$ to be the homotopy-colimit (\ie derived colimit)
%
%
\[ \hocolim_e F^e_* K := \hocolim\Big(K \stackrel{\phi_K}{\to} F_* K \stackrel{F_*(\phi_K)}{\to} F^2_* K \to ... \Big)\]
as in \cite[Tag 0A5K]{stacks-project}.  
Then $K_\infty \in D(R_\infty)$. Note that we have
\[ H^i(K_\infty) := \varinjlim_e F^e_* H^i(K) := \varinjlim\Big(H^i(K) \stackrel{\phi_K}{\to} F_* H^i(K) \stackrel{F_*(\phi_K)}{\to} F^2_* H^i(K) \to ... \Big),\]
as cohomology commutes with direct limits. Also, the canonical map $K \to K_\infty$ is injective on each $H^i(-)$ by the assumption on $\phi_K$ (and because $F^e_*$ is exact). We claim that $K_\infty$ lies in the essential image of the forgetful functor $D(k_\infty) \to D(R_\infty)$; this suffices to prove the result by \autoref{lem:ComplexofVectorSpaces}. To see this, using \autoref{prop:PerfectFullyFaithfulRoS} for the map $R_\infty \to k_\infty$, it is enough to show that $H^i(K_\infty)$ is killed by $\m_\infty \subset R_\infty$. By the formula for $H^i(K_\infty)$ above and the formula $\m_\infty = \varinjlim_e F^e_* \m$ (with the transition maps being Frobenius), it suffices to show that $\m$ kills $H^i(K)$. As $K$ has finite length homology, there exists some $c$ such that $\m^c$ kills $H^i(K)$. But then $\m$ kills $F^e_* H^i(K)$ if $e \geq \log_p(c)$ since $F^e(\m) \subset \m^c$ for such $e$.  As the map $H^i(K) \to F^e_* H^i(K)$ induced by an $e$-fold iterate of $\phi_K$ is injective, it then follows that $\m$ kills $H^i(K)$, as wanted.
\end{proof}

This lemma specializes to prove the following result:

\begin{theorem}
\label{theorem--main theorem on F-injective}
Let $(R,\m,k)$ be a Noetherian local ring of characteristic $p>0$.  Assume there exists some $j \geq 0$ such that $H^i_\m(R)$ has finite length for every $i < j$, and that the Frobenius on $R$ induces an injective map on $H^i_\m(R)$ for $i < j$ (e.g., $R$ is $F$-injective). Then $\tau^{< j} \myR\Gamma_\m(R) \in D(R)$ comes from an object of $D(k)$. In other words, $\tau^{< j} \myR\Gamma_\m(R)$ and $\tau_{>-j}\omega_R^\mydot$ are quasi-isomorphic to complexes of $k$-vector spaces.
\end{theorem}
\begin{proof}
Set $K = \tau^{< j} \myR\Gamma_\m(R)$. Then the Frobenius on $R$ induces a map $\phi_K:K \to F_* K$ that satisfies the hypotheses of \autoref{lem:PhiModuleCriterion} by assumption. Thus by \autoref{lem:PhiModuleCriterion} $\tau^{< j} \myR\Gamma_\m(R)$ comes from $D(k)$.
\end{proof}

As an immediate consequence of the above theorem, we reprove and in fact generalize the main result of \cite{MaF-injectivityandBuchsbaumsingularities}. We note that this result is also obtained using different methods in \cite{QuyShimomotoFinjectivityandFrobeniusclosure}.

\begin{corollary}
Let $(R,\m, k)$ be a Noetherian local ring of characteristic $p>0$ and dimension $d$ such that $H^i_\m(R)$ has finite length for every $i < d$. Suppose Frobenius acts injectively on $H_\m^i(R)$ for every $i<d$ (e.g., $R$ is $F$-injective). Then $R$ is Buchsbaum.
\end{corollary}
\begin{proof}
Applying \autoref{theorem--main theorem on F-injective} to $j=d$, $\tau^{< d} \myR\Gamma_\m(R)$ is quasi-isomorphic to a complex of $k$-vector spaces. This implies $R$ is Buchsbaum by Schenzel's criterion \cite{SchenzelApplicationsOfDualizingComplexes}.
\end{proof}

Next we prove a stronger version of \autoref{theorem--main theorem on F-injective} when $R$ is $F$-rational on the punctured spectrum. Basically we will show that, in this case, if we truncate $\myR\Gamma_\m R$ at the $d$-th spot (resp., $\omega_R^\mydot$ at the $-d$-th spot) ``up to tight closure", it is still quasi-isomorphic to a complex of $k$-vector spaces.

We introduce some notations. We let $(R,\m, k)$ be a reduced and equidimensional local ring of characteristic $p>0$ and dimension $d$. Let $\myR\Gamma_\m R= 0\to G^0\to G^1\to \cdots\to G^{d-1}\xrightarrow{\phi} G^d\to 0.$ We define
\[
\tau^{<d,*}\myR\Gamma_\m R=0\to G^0\to G^1\to \cdots\to G^{d-1}\xrightarrow{\phi} (\im\phi)^*_{G^d}\to 0
 \]
to be the {\it ${}^*$-truncation} of $\myR\Gamma_\m R$ at the $d$-th spot (where $(-)^*$ denotes tight closure). This is a well-defined object in the derived category of $R$-modules: it is the natural object such that we have an exact triangle:\footnote{Equivalently, $\tau^{< d,*} \myR\Gamma_\m(R)[1] \in D(R)$ is the cone of the canonical composite map $\myR\Gamma_\m(R) \to H^d_\m(R)[-d] \to (H^d_\m(R)/0^*_{H^d_\m(R)})[-d]$.}
$$\tau^{<d,*}\myR\Gamma_\m R\to\myR\Gamma_\m R\to H_\m^d(R)/0^*_{H_\m^d(R)}[-d]\xrightarrow{+1}.$$

Dually, let $\omega_R^\mydot=0\to I^{-d}\to I^{-d-1}\to\cdots \to I^{-1}\to I^0\to 0$ be the normalized dualizing complex and let $\tau(\omega_R)\subseteq \omega_R$ be the parameter test submodule of $R$ (see \cite{Smithtightclosureofparameterideals} or Section 2 of \cite{BlickleSchwedeTuckerF-singularitiesvialterations} for definitions and details). We define \[
\tau^*_{>-d}\omega_R^\mydot=0\to I^{-d}/\tau(\omega_R)\to I^{-d-1}\to \cdots \to I^{-1}\to I^0\to 0
 \]
to be the ${}^*$-truncation of $\omega_R^\mydot$ at $-d$-th spot. Again this is the object in the derived category that completes the triangle:
$$\tau(\omega_R)[d] \to \omega_R^{\mydot} \to \tau^*_{>-d}\omega_R^\mydot \xrightarrow{+1}.$$

It is easy to check that $\tau^{<d,*}\myR\Gamma_\m R$ is the Matlis dual of $\tau^*_{>-d}\omega_R^\mydot$, as $\tau(\omega_R)^\vee\cong H_\m^d(R)/0^*_{H_\m^d(R)}$. We assume the readers are familiar with basic tight closure theory \cite{HochsterHunekeTC1}, in particular the tight closure of $0$ in $H_\m^d(R)$, see \cite{Smithtightclosureofparameterideals}, \cite{SmithFRatImpliesRat}, \cite{EnescuHochsterTheFrobeniusStructureOfLocalCohomology}.

\begin{theorem}
\label{theorem--main theorem on F-injective tight closure}
Let $(R,\m, k)$ be an excellent equidimensional Noetherian local ring of characteristic $p > 0$ and dimension $d$. Assume that $R$ is $F$-injective, and also that $\mathrm{Spec}(R) - \{\m\}$ is $F$-rational. Then $\tau^{< d,*} \myR\Gamma_\m(R)$ and $\tau^*_{>-d}\omega_R^\mydot$ come from $D(k)$, i.e., they are quasi-isomorphic to complexes of $k$-vector spaces.
\end{theorem}
\begin{proof}
Set $K = \tau^{< d,*} \myR\Gamma_\m(R)$. Then the Frobenius on $R$ induces a map $\phi_K:K \to F_* K$. We claim that this satisfies the hypotheses of \autoref{lem:PhiModuleCriterion}, which suffices to prove the result. To see this, note that $H^i(K) = H^i_\m(R)$ for $i \neq d$, and $H^d(K) = 0^*_{H^d_\m(R)}$. Now the punctured $F$-rationality assumption implies that $H^i(K)$ has finite length for all $i$, while the $F$-injectivity assumption ensures that $H^i(\phi_K)$ is injective for all $i$.
\end{proof}

\subsection{Alternative proofs of the main theorems when $k$ is perfect}
\label{subsec.altProofs}
In this subsection we give a more elementary, but in some ways more involved, proof of \autoref{theorem--main theorem on F-injective}.  We start with some general lemmas.

\begin{lemma}
\label{lemma--duality of the maps}
Let $(A,\m)\to (S,\n)$ be a module-finite ring homomorphism with $A$ Noetherian and regular. Then the map $\Ext_A^i(k, S)\to H_\m^i(S)$ is the Matlis dual of the map $\myH^{-i}(\omega_S^\mydot)\to \myH^{-i}(\omega_S^\mydot\otimes_A^{\mathbf{L}}k)$.
\end{lemma}
\begin{proof}
First notice that $$\mathbf{R}\Hom_A(k, S)= \mathbf{R}\Hom_A(k, \mathbf{R}\Hom_A(\omega_S^\mydot, \omega_A^\mydot))=\mathbf{R}\Hom_A(k\otimes_A^\mathbf{L}\omega_S^\mydot, \omega_A^\mydot).$$ Since $A$ is regular, $k\otimes_A^\mathbf{L}\omega_S^\mydot$ is a bounded complex. By local duality, $\myH^{i}(\mathbf{R}\Hom_A(k\otimes_A^\mathbf{L}\omega_S^\mydot, \omega_A^\mydot))$ is the Matlis dual of $\myH^{-i}(\mathbf{R}\Gamma_\m(\omega_S^\mydot\otimes_A^{\mathbf{L}}k))=\myH^{-i}(\omega_S^\mydot\otimes_A^{\mathbf{L}}k)$.

The map $\Ext_A^i(k, S)\to H_\m^i(S)$ can be identified with the natural map $$\myH^{i}(\mathbf{R}\Hom_A(k, S))\to\myH^{i}(\varinjlim_{n}\mathbf{R}\Hom_A(A/\m^n, S))\cong \myH^{i}(\mathbf{R}\Gamma_\m S).$$ Hence by local duality this is the Matlis dual of $ \myH^{-i}(\omega_S^\mydot\otimes_A^{\mathbf{L}}k)  \leftarrow \myH^{-i}(\omega_S^\mydot).$
\end{proof}

\begin{lemma}
\label{lemma--injectivity on dualizing complex for large e}
Let $(A,\m, k)\twoheadrightarrow (R,\m,k)$ be a surjective homomorphism between $F$-finite Noetherian local rings  of characteristic $p>0$. Let $G^\mydot$ be a bounded complex of $R$-modules such that $\myH^{-i}G^\mydot$ has finite length. If $A$ is regular, then $\myH^{-i}(F^e_*G^\mydot)\to \myH^{-i}(F^e_*G^\mydot\otimes_A^{\mathbf{L}}k)$ is injective for $e\gg 0$.
\end{lemma}
\begin{proof}
By the projection formula, it is enough to show $\myH^{-i}(G^\mydot)\to \myH^{-i}(G^\mydot\otimes_A^{\myL}A/\m^{[p^e]})$ is injective for $e\gg 0$.
Set $\dim A=n$. We take a minimal free resolution of $A/\m^{[p^e]}$ over $A$: $$P^\mydot=0\to P^{-n}\to P^{-n+1}\to\cdots\to P^{-1}\xrightarrow{g} P^0(=A)\to 0,$$ where we use $P^{-i}$ to emphasize that these modules are in cohomology degree $-i$. The map $g$ is represented by a minimal generating set of $\m^{[p^e]}$, say $[y_1^{p^e},\dots,y_n^{p^e}]$. We consider the double complex $G^\mydot\otimes_A P^\mydot$:
\[\xymatrix{
{} & \vdots \ar[d] & \vdots \ar[d] & \vdots\ar[d] & {}\\
\cdots \ar[r] & G^{-i-1}\otimes P^{-1} \ar[r]\ar[d]^{\id\otimes g} & G^{-i}\otimes P^{-1} \ar[r]\ar[d]^{\id\otimes g}& G^{-i+1}\otimes P^{-1} \ar[r]\ar[d]^{\id\otimes g} & \cdots \\
\cdots \ar[r] & G^{-i-1} \ar[r]^{\alpha_{i+1}}  & G^{-i} \ar[r]^{\alpha_i} & G^{-i+1} \ar[r] & \cdots
}\]

We know that the map $\myH^{-i}(G^\mydot)\to \myH^{-i}(G^\mydot\otimes_A^{\myL}A/\m^{[p^e]})$ is induced by the natural map of the $-i$-th cohomology of the bottom row to the $-i$-th cohomology of the total complex of $G^\mydot\otimes_A P^\mydot$. Pick $x\in\ker\alpha_i$, if $(x,0,\dots,0)\in\oplus_s(G^{-i+s}\otimes P^{-s})$ is a boundary in the total complex, then we have:
\begin{eqnarray*}
\label{equation--induced map from the double complex}
x &\in & (\im\alpha_{i+1}+\im(\id\otimes g))\cap\ker{\alpha_i} \\
 &=& (\im\alpha_{i+1}+\m^{[p^e]}G^{-i})\cap\ker{\alpha_i} \\
 &\subseteq& (\im\alpha_{i+1}+\m^{p^e}G^{-i})\cap\ker{\alpha_i}
\end{eqnarray*}
Hence for $e\gg 0$, by Artin-Rees lemma, we have $x\in \im\alpha_{i+1}+\m^{p^e-l}\ker{\alpha_i}$ for some fixed $l>0$. Since $\myH^{-i}(G^\mydot)=\ker\alpha_i/\im\alpha_{i+1}$ has finite length by assumption, for $e\gg0$ we have $\m^{p^e-l}\ker{\alpha_i}\subseteq\im\alpha_{i+1}$. Thus for $e\gg0$, $x\in \im\alpha_{i+1}$, that is, $(x,0,\dots,0)\in\oplus_s(G^{-i+s}\otimes P^{-s})$ is a boundary in the total complex if and only if $x$ is already a boundary in $G^\mydot$. This proves $\myH^{-i}(G^\mydot)\to \myH^{-i}(G^\mydot\otimes_A^{\myL}A/\m^{[p^e]})$ is injective for $e\gg0$.
\end{proof}

The following lemma is crucial. The idea comes from \cite[Proposition 6.3.5]{HanesThesis}

\begin{lemma}
\label{lemma--depth of F^e_*R/R}
Let $(R,\m,k)$ be an $F$-finite Noetherian local ring of characteristic $p>0$ such that $H_\m^i(R)$ has finite length for every $0\leq i <j$. If $R$ is $F$-injective and $k$ is perfect, then $\depth((F^e_*R)/R)\geq j$ for every $e\geq 1$.
\end{lemma}
\begin{proof}
First of all, $R$ is $F$-finite and $F$-injective, which implies $R$ is reduced, for example see \cite{SchwedeandZhangBertinitheoremsforFsingularities}. Now the short exact sequence: $$0\to R\to F_*^eR\to (F_*^eR)/R\to 0$$ induces a long exact sequence on local cohomology:
$$\cdots \to H_\m^i(R)\xrightarrow{\phi_i} H_\m^i(F_*^eR)\to H_\m^i((F_*^eR)/R)\to \cdots$$
Since $k$ is perfect, we have $l_R(H_\m^i(R))=l_{F^e_*R}(H_\m^i(F_*^eR))=l_R(H_\m^i(F_*^eR))$. Since $R$ is $F$-injective, for every $0\leq i<j$, the map $\phi_i$ is an injective map between $R$-modules of the same length, and hence bijective. Chasing the long exact sequence we then know that $H_\m^i((F_*^eR)/R)=0$ for every $0\leq i<j$, which proves that $\depth(F_*^eR)/R\geq j$.
\end{proof}

Now we give an alternative proof of \autoref{theorem--main theorem on F-injective} when $R$ is $F$-injective and $k$ is perfect.

\begin{theorem}
Let $(R,\m,k)$ be a Noetherian local ring of characteristic $p>0$ such that $H_\m^i(R)$ has finite length for every $0\leq i <j$. If $R$ is $F$-injective with $k$ perfect, then $\tau_{>-j}\omega_R^\mydot$ is quasi-isomorphic to a complex of $k$-vector spaces.
\end{theorem}
\begin{proof}
Since whether $\tau_{>-j}\omega_R^\mydot$ or $\tau^{<j}\myR\Gamma_\m R$ is quasi-isomorphic to a complex of $k$-vector spaces is unaffected under completion, we may complete $R$ and assume $(R,\m,k)$ is a complete local ring with $k$ perfect. In particular, $R$ is $F$-finite and by Cohen's structure theorem we can fix $(A,\m, k)\twoheadrightarrow (R,\m, k)$ with $A$ a regular local ring. By \autoref{lemma--injectivity on dualizing complex for large e} applied to $G^\mydot=\omega_R^\mydot$, we can pick $e\gg 0$ such that $\myH^{-i}(F^e_*\omega_R^\mydot)\to \myH^{-i}(F^e_*\omega_R^\mydot\otimes_A^{\mathbf{L}}k)$ is injective for every $0\leq i<j$ (note that $\myH^{-i}\omega_R^\mydot$ has finite length because its dual $H_\m^i(R)$ has finite length for every $0\leq i <j$). Now by \autoref{lemma--duality of the maps} applied to $S=F^e_*R$, we have $\Ext_A^i(k, F_*^eR)\to H_\m^i(F_*^eR)$ is surjective for every $0\leq i<j$.

By \autoref{lemma--depth of F^e_*R/R} we know that $\depth_A(F_*^eR)/R\geq j$. Now the short exact sequence $$0\to R\to F_*^eR\to (F_*^eR)/R\to 0$$ induces the following commutative diagram:
\[\xymatrix{
0=\Ext_A^{i-1}(k, (F_*^eR)/R) \ar[r] & \Ext_A^i(k, R) \ar[r] \ar[d] & \Ext_A^i(k, F_*^eR) \ar[r]\ar@{->>}[d] & \Ext_A^i(k, (F_*^eR)/R)=0 \\
0=H_\m^{i-1}((F_*^eR)/R)\ar[r] & H_\m^i(R) \ar[r] & H_\m^i(F_*^eR) \ar[r] & H_\m^i((F_*^eR)/R)=0
}\]
A diagram chase shows that the surjectivity of $\Ext_A^i(k, F_*^eR)\to H_\m^i(F_*^eR)$ implies the surjectivity of $\Ext_A^i(k, R)\to H_\m^i(R)$. Thus \autoref{lemma--duality of the maps} implies that $\myH^{-i}(\omega_R^\mydot)\to \myH^{-i}(\omega_R^\mydot\otimes^\myL_A k)$ is injective and thus so is $\myH^{-i}(\tau_{>-j}\omega_R^\mydot)\to \myH^{-i}(\tau_{>-j}\omega_R^\mydot\otimes^\myL_A k)$. Now \autoref{lem:ComplexofVectorSpaces} shows that $\tau_{>-j}\omega_R^\mydot$ is quasi-isomorphic to a complex of $k$-vector spaces.
\end{proof}

\begin{remark}
In the previous version of this paper (available on the arXiv) we gave a proof of \autoref{theorem--main theorem on F-injective tight closure} again in the case where the residue field is perfect.  However, the proof was rather involved and so we omit it here.
\end{remark}

\section{Du~Bois singularities}

We now prove our main results for Du~Bois singularities.  Suppose $X$ is a reduced scheme essentially of finite type over a field $k$ of characteristic zero and $\pi : W \to X$ is a resolution of singularities.  Then by for instance \cite[Theorem 4.2.2]{KovacsSchwedeDuBoisSurvey}, there exists a natural map $\DuBois{X} \to \myR \pi_* \DuBois{W}$.  Since $W$ is smooth, and hence has Du~Bois singularities, $\DuBois{W} \qis \O_W$ and hence there is a natural map
\begin{equation}
\label{eqn.DuBoisToRes}
\DuBois{X} \to \myR \pi_* \O_W.
\end{equation}

\begin{theorem}
\label{thm.MainTheoremOnDuBois}
Suppose that $(R, \bm, k)$ is a reduced local ring essentially of finite type over a field of characteristic zero.  Let $\kappa : W \to X = \Spec R$ be a resolution of singularities.  Suppose that
\[
\myR \kappa_* \omega_W^{\mydot} \to \myR \sHom_X(\DuBois{X}, \omega_{X}^{\mydot}) \to C \xrightarrow{+1}
\]
is a triangle where the first map is the Grothendieck dual of \autoref{eqn.DuBoisToRes}.  Suppose that $\myH^{-i} C$ has finite length for all $0 \leq i < j$ for some $j$ (equivalently, it is supported within $V(\m)$ for those same $i$).  Then $\tau_{> -j} C$ is quasi-isomorphic to a complex of $k$-vector spaces and in particular, $\tau_{> -j} C \qis \oplus_{0 \leq i < j} \myH^{-i}(C)[i]$ is a direct sum of its cohomologies.
\end{theorem}
\begin{proof}
Let $(A,\m)\to(R,\m)$ be a surjection with $A$ regular. Let $X=\Spec R$ and $Y=\Spec A$. We take a log resolution $\pi : \tld Y \to Y$ of $(Y, X)$ obtained in the following way.  First form an embedded resolution of $X$ in $Y$, $\kappa : Y' \to Y$, in such a way that the exceptional divisor is SNC and intersects the strict transform of $X$ with normal crossings in a SNC divisor.  Then blow up the strict transform $W$ of $X$ in $Y'$, see for instance \cite{BravoEncinasVillmayorSimplified}, obtaining a log resolution
\[
\pi : \tld Y \xrightarrow{\rho} Y' \xrightarrow{\kappa} Y
\]
 of $(Y, X)$ and also of $(Y', W)$.  Note $\kappa = \kappa|_W: W \to X$ is a resolution of singularities.  We observe also that $\tld X = \rho^{-1} W$ is the union of components of $\overline{X} = \pi^{-1}(X)_{\red}$ whose $\pi$-image dominates a component of $X$.   Write $\overline{X} = \tld X + E + F$ where $E = \pi^{-1}(V(\m))_{\reduced}$.
\begin{claim}With notation as above $\myR \pi_* \omega_{\tld X}^{\mydot} \qis \myR \kappa_* \omega_{W}^{\mydot}$.  \end{claim}
\begin{proof}[Proof of claim]
There is a map $\rho : \tld X \to W$ by construction.  It suffices to show that $\myR \rho_* \omega_{\tld X}^{\mydot} \qis \omega_{W}^{\mydot}$.  But $W$ is smooth and hence Du~Bois and recall that $(\tld Y, \tld X = (\rho^{-1} W)_{\reduced})$ is a log resolution of $(Y', W)$.  Therefore by applying the criterion from \autoref{def.DuBoisSings}, we see that $\myR \rho_* \O_{\tld X} \qis \O_W$.  Applying Grothendieck duality proves the claim.
\end{proof}
We next note that $\myR \pi_* \omega_{\overline{X}}^{\mydot} \qis \myR \sHom_X(\DuBois{X}, \omega_{X}^{\mydot})$.  Further observe that the map
\[
\myR \pi_* \omega_W^{\mydot} \to \myR \sHom_X(\DuBois{X}, \omega_{X}^{\mydot})
\]
is induced by the inclusion $\tld{X} \hookrightarrow \overline{X}$.    By blowing up further (even at the $\kappa$-stage) if needed, we may assume that no stratum of $\tld{X} + F$ lies over $\m$.  Indeed, if there are any such strata, simply blow them up so that their inverse images become parts of $E$.

With this notation, we observe that we have the short exact sequence:
\[
0 \to \O_{E \cup F}(-\tld{X}) \to \O_{\overline{X}} \to \O_{\tld{X}} \to 0
\]
and so by dualizing and pushing forward, we see that
\[
C \qis \myR \pi_* \omega_{E \cup F}^{\mydot}(\tld X).
\]
From the short exact sequence
\[
0 \to \O_{E}(-\tld X - F) \to \O_{E \cup F}(-\tld X) \to \O_{F}(-\tld X) \to 0
\]
we obtain
\[
\myR \pi_* \omega_F^{\mydot}(\tld X) \to C \to \myR \pi_* \omega_{E}^{\mydot}(\tld X + F) \xrightarrow{+1}.
\]
By our initial definition of $C$, we see that $\Supp \myH^{-i} C \subseteq V(\m)$ for every $i < j$.  Since
\[
\Supp \myH^{-i} \myR \pi_* \omega_{E}^{\mydot}(\tld X + F) \subseteq V(\m)
\]
for every $i$ (since $\pi(E) = V(\m)$), we also conclude that
\[
\Supp \myH^{-i}\myR \pi_* \omega_F^{\mydot}(\tld X) \subseteq V(\m)
\]
for every $i < j$.  However, observe that $(F, \tld X|_F)$ is a semi-log canonical pair (since it is SNC), and hence that any associated prime of $\myH^{-i} \myR \pi_* \omega_F(\tld X)$ is the image of a (slc-)strata of $(F, \tld X|_F)$ by \cite[Theorem 3.2]{AmbroQuasiLogVarieties} or \cite[Theorem 1.12]{FujinoFundamentaltheoremforSLC}.  But $V(\m)$ is not any such image by our construction. Hence we see that $\myH^{-i}\myR \pi_* \omega_F^{\mydot}(\tld X) = 0$ for all $i < j$.  Thus $\tau_{> -j} \myR \pi_* \omega_F^{\mydot}(\tld X) \qis 0$.  It follows that
\[
\tau_{> -j} C \qis \tau_{> -j} \myR \pi_* \omega_{E}^{\mydot}(\tld X + F).
\]
But since $\pi(E) = V(\m)$, we see that $\m$ annihilates the terms of $\myR \pi_* \omega_{E}^{\mydot}(\tld X + F)$.  In other words, $\myR \pi_* \omega_{E}^{\mydot}(\tld X + F)$ is a complex of $k$-vector spaces.  Thus $\tau_{>-j} C$ is quasi-isomorphic to a complex of $k$-vector spaces as claimed.
\end{proof}

While technical, this theorem immediately implies results analogous to those we proved previously for $F$-injective singularities.

\begin{corollary}
\label{corollary--main corollary on DB}
Suppose that $(R, \m, k)$ is a reduced local ring essentially of finite type over a field of characteristic zero.  Suppose that $X = \Spec R$ has Du Bois singularities and that $H^i_{\m}(R)$ has finite length for all $i < j \leq \dim R$.   Then $\tau^{< j} \myR \Gamma_{\m}(R)$ is quasi-isomorphic to a complex of $k$-vector spaces or dually, $\tau_{> -j} \omega_X^{\mydot}$ is quasi-isomorphic to a complex of $k$-vector spaces.

In particular, Du~Bois singularities with isolated non-Cohen-Macaulay locus are Buchsbaum.
\end{corollary}
\begin{proof}
Because $X$ has Du~Bois singularities, $\omega_X^{\mydot} \qis \myR \sHom_X(\DuBois{X}, \omega_{X}^{\mydot})$.  We then notice that $\tau_{> -j} C \qis \tau_{> -j} \omega_X^{\mydot}$ where $C$ is defined as in \autoref{thm.MainTheoremOnDuBois}, since $j \leq \dim R$.  The result immediately follows.
\end{proof}

Since the parameter test submodule $\tau({\omega_R})$ is well-known to be the characteristic $p>0$ analogue of the Grauert-Riemenschneider canonical sheaf $\pi_*\omega_W=\omega_R^{GR}$. The following result is an analogue of \autoref{theorem--main theorem on F-injective tight closure}.

\begin{corollary}
\label{corollary--main corollary on DB punctured rational}
Suppose that $(R, \m, k)$ is a reduced and equidimensional local ring essentially of finite type over a field of characteristic zero.  Suppose that $X = \Spec R$ has Du~Bois singularities and that $X \setminus V(\m)$ has rational singularities.  Let $\pi : W \to X$ is a resolution of singularities and consider the triangle:
\[
\pi_* \omega_W[d] \to \omega_X^{\mydot} \to C \xrightarrow{+1}.
\]
Then $C$ is quasi-isomorphic to a complex of $k$-vector spaces and so is a direct sum of its cohomologies.
\end{corollary}
\begin{proof}
We first observe that $\pi_* \omega_W[d] \qis \myR \pi_* \omega_W^{\mydot}$ by the Grauert-Riemenschneider vanishing theorem because $W$ is smooth.  Since $X \setminus V(\m)$ has rational singularities, we see that $\myH^{-i} C$ is supported within $V(\m)$ for all $i < j=\dim X+1$.  Finally since $X$ has Du~Bois singularities, $\O_X \qis \DuBois{X}$, and we see that $\omega_X^{\mydot} \qis \myR \sHom_X(\DuBois{X}, \omega_{X}^{\mydot})$.  Putting this all together, we see that $C$ is the $C$ as in \autoref{thm.MainTheoremOnDuBois} and that all the conditions of \autoref{thm.MainTheoremOnDuBois} are satisfied with $j=\dim X+1$ (and hence $C=\tau_{>-j} C$). The conclusion follows.
\end{proof}

\section{Further questions}

Buchsbaum singularities have the property that their truncated $(> -\dim)$ dualizing complexes are the direct sum of their cohomologies.  One might ask if the full dualizing complex splits into a direct sum of all its cohomologies, especially in view of results such as \cite{KollarHigherdirectimagesofdualizingsheaves1} or \cite[Theorem 3.2]{KollarKovacsLCImpliesDB}.  However, without the truncation, this is not true.  The only way a dualizing complex can be a direct sum of its cohomologies is if the ring is Cohen-Macaulay. In fact, there is a more general statement concerning the indecomposability of local cohomology complexes:

\begin{proposition}
\label{proposition--local cohomology complex are indecomposable}
Let $R$ be a Noetherian ring, and $I \subset R$ an ideal such that $\mathrm{Spec}(R/I)$ is connected. Then $\myR\Gamma_I(R)$ is indecomposable in $D(R)$, i.e., it admits no non-trivial direct sum decomposition.
\end{proposition}

In the proof below, we write $D_I(R)$ for the full subcategory of $D(R)$ spanned by complexes $K$ such that $K$ restricts to $0$ in $D(U)$, where $U = \mathrm{Spec}(A) - \mathrm{Spec}(R/I)$; this is equivalent to asking that each $H^i(K)$ is $I^\infty$-torsion.

\begin{proof}
Write $\widehat{R}$ for the $I$-adic completion of $R$, and let $K \mapsto \widehat{K}$ denote the derived $I$-adic completion functor $D(R) \to D(\widehat{R})$; as $R$ is Noetherian, we can calculate $\widehat{K}$ by applying the naive $I$-adic completion functor $M \mapsto \lim M/I^n M$ to terms of a flat representative $K^\mydot$ of $K$ in $D(R)$ (see \cite[Tag 091N]{stacks-project}). In particular, the derived $I$-adic completion of $R$ is $\widehat{R}$, so the notation is consistent. We need the following fact: if $K \in D_I(R)$, then $\myR\Gamma_I(\widehat{K}) \simeq K$; see \cite[Tag 06AV]{stacks-project} for a proof. In particular, the derived $I$-adic completion of any non-zero object in $D_I(R)$ is non-zero, so this functor takes decomposable objects of $D_I(R)$ to decomposable objects of $D(\widehat{R})$.

Now consider $K = \myR\Gamma_I(R) \in D_I(R)$. Using the Cech complex to calculate $\myR\Gamma_I(R)$, we see that $\widehat{K} \simeq \widehat{R}$. Now $\widehat{R} \in D(\widehat{R})$ is indecomposable by the hypothesis that $\mathrm{Spec}(R/I)$ is connected: if $\widehat{R} \simeq M \oplus N$ in $D(\widehat{R})$, then each of $M$ and $N$ are finite projective $\widehat{R}$-modules (viewed as complexes placed in degree $0$), and their ranks give locally constant $\mathbb{N}$-valued functions $r_M,r_N:\mathrm{Spec}(\widehat{R}) \to \mathbb{N}$ such that $r_M + r_N = 1$. As $\mathrm{Spec}(R/I)$ is connected, so is $\mathrm{Spec}(\widehat{R})$, so all such locally constant functions are constant, so either $r_M = 1$ and $r_N = 0$ or vice versa.  Thus, either $M \simeq \widehat{R}$ and $N = 0$, or vice versa. Now the fact quoted in the previous paragraph implies that $\myR\Gamma_I(R)$ is also indecomposable, proving the claim.
\end{proof}

As a special case of \autoref{proposition--local cohomology complex are indecomposable} (the case $I=\m$), we have the following result which we expect is well known to experts but we do not know a reference.

\begin{corollary}
\label{corollary--full dualizing complex never splits}
Let $(R,\m)$ be a Noetherian local ring of dimension $d$ with a dualizing complex $\omega_R^\mydot$. Then $\omega_R^\mydot\qis \oplus_{i=0}^d(\myH^{-i}\omega_R^\mydot)[i]$ if and only if $R$ is Cohen-Macaulay. Equivalently, $\myR\Gamma_\m R \qis\oplus_{i=0}^d(H_\m^i(R))[-i]$ if and only if $R$ is Cohen-Macaulay.
\end{corollary}

This shows that one at least must truncate at the bottom degree of the dualizing complex.  Of course, there are numerous rings $R$ whose truncated dualizing complexes split into direct sums of their cohomologies even if the non-Cohen-Macaulay locus is not isolated.  For instance, suppose that $R$ is an $F$-injective or Du Bois $d$-dimensional ring whose non-Cohen-Macaulay locus is isolated.  Let $S = R[x]$.  Then the map $f : \Spec S \to \Spec R$ is smooth and hence $\omega_S^{\mydot} = f^! \omega_R^{\mydot} = f^* \omega_R^{\mydot} = S \otimes_R \omega_R^{\mydot}$, see for instance \cite[Chapter VII, Section 4]{HartshorneResidues}.  It follows that the non-Cohen-Macaulay locus of $S$ is not isolated, but that $\tau_{>-d-1} \omega_{S}^{\mydot}$ is a direct sum of its cohomologies.  This suggests the following question.

\begin{question}
Do any related conditions on singularities (for instance, $F$-injective, $F$-pure, Du~Bois, log canonical, having only a unique $F$-pure or log canonical center) imply that the truncated dualizing complex
\[
\tau_{> -d} \omega_R^{\mydot}
\]
is a direct sum of its cohomologies?  In particular, does this behavior still occur even when the non-Cohen-Macaulay locus is not isolated?
\end{question}

Our characteristic zero result on \autoref{thm.MainTheoremOnDuBois} actually suggests the following question as well.

\begin{question}
Suppose that $(R, \m, k)$ is an $F$-finite Noetherian local ring of characteristic $p > 0$.  Is there a complex $B \in D^b(R)$ analogous to the characteristic zero object $\myR \Gamma_{\m}(\DuBois{X})$, the Matlis dual of $\myR \Hom_X(\DuBois{X}, \omega_X^{\mydot})$.  In particular, we would hope for the following properties (although some might be weakened).
\begin{enumerate}
\item We have a functorial map $\myR \Gamma_{\m}(R) \to B$ in the derived category.
\item The induced maps on cohomology $H^i_{\m}(R) \to \myH^{i}(B)$ are surjective with kernel equal to $0^{F}_{H^i_{\m}(R)}$, the Frobenius closure of zero.
\item If $D \to B \to (H^d_{\m}(R) / 0^*_{H^d_{\m}(R)})[-d] \xrightarrow{+1}$ is an exact triangle and $\myH^{i}(D)$ has finite length for all $i$, then $D$ is quasi-isomorphic to a complex of $k$-vector spaces.\label{FinalQuestionBit}
\end{enumerate}
The statement \autoref{FinalQuestionBit} is the analog of \autoref{thm.MainTheoremOnDuBois}. If the residue field $k$ is perfect and each $H^i_\m(R)$ has finite length for $i < \dim(R)$, then such a $B$ exists: we may set $B$ to be the ``homotopy fibre product''  $\myR\Gamma_\m(R_\infty) \times_{H^d_\m(R_\infty)[-d]} (H^d_\m(R)/0^F_{H^d_\m(R)})[-d]$ (defined formally as a suitable shifted cone). In general, however, we do not know if such a $B$ exists.
\end{question}

\bibliographystyle{skalpha}
\bibliography{CommonBib}
\end{document}